\documentclass[11pt]{article}
\usepackage{amsmath}
\usepackage{amsfonts}
\usepackage{amssymb}

\usepackage{ntheorem,lipsum}
\theorembodyfont{\upshape}
\newtheorem{theorem}{Theorem}
\newtheorem{corollary}{Corollary}
\newtheorem{definition}{Definition}
\newtheorem{example}{Example}
\newtheorem{lemma}{Lemma}
\newtheorem{proposition}{Proposition}
\newtheorem{remark}{Remark}
\newenvironment{proof}[1][Proof]{\noindent\textbf{#1.} }{\ \rule{0.5em}{0.5em}}

\begin{document}

\title{\textbf{ Lieb, Entropy and Logarithmic uncertainty principles
 for The multivariate Continuous\\
  Quaternion Shearlet Transform}}
\author{ Brahim Kamel
\thanks{  Department of Mathematics, College of Science, University of Bisha, Bisha, Saudi Arabia.\newline E-mail:
kbrahim@ub.edu.sa}
\thanks{ Faculty of Sciences of Tunis. University of Tunis El Manar, Tunis, Tunisia.\newline  E-mail :
kamel.brahim@fst.utm.tn}
\qquad \& \quad Emna Tefjeni
\thanks{ Faculty of Sciences of Tunis. University of Tunis El Manar, Tunis, Tunisia.\newline  E-mail :
tefjeni.emna@outlook.fr}
\qquad \& \quad  Bochra Nefzi
\thanks{
 Department of Mathematics, College of Science and Arts - Tabarjal, Al Jouf University , Al Jouf,  Kingdom of Saudi Arabia.\newline  E-mail
:
bochra.nefzi@fst.utm.tn}}
\date{}
\maketitle
\begin{abstract}In this paper, we generalize the continuous quaternion shearlet transform on $\mathbb{R}^{2}$ to $\mathbb{R}^{2d}$, called the multivariate two sided continuous quaternion shearlet transform. Using the two sided quaternion Fourier transform, we derive several important properties such as (reconstruction formula, reproducing kernel, plancherel's formula, etc.). We present several example of the multivariate two sided continuous quaternion shearlet transform. We apply the multivariate two sided continuous quaternion shearlet transform properties and the two sided quaternion Fourier transform to establish Lieb uncertainty principle and the Logarithmic uncertainty principle. Last we study   the Beckner's uncertainty principle in term of entropy for the multivariate two sided continuous quaternion shearlet transform.
\end{abstract}
\vspace{2mm}
\section{Introduction}
Uncertainty principles are mathematical results that give limitations on the simultaneous concentration of a function and its Euclidean Fourier transform. They have implications in two main areas: quantum physics and signal analysis. In quantum physics, they tell us that a particle's speed and position cannot both be measured with infinite precision. In signal analysis, they tell us that if we observe a signal only for a finite period of time, we will lose information about the frequencies the signal consists of. There are many ways to get the statement about concentration precise. For more details about uncertainty principles, we refer the reader to \cite{uncertainty1, uncertainty2}.\\
Labate et al. \cite{Labate} introduced the notion of shearlet transforms in the context of time frequency and multiscale analysis. Shearlets
have been applied in diverse areas of engineering and medical sciences, including inverse problems, computer tomography, image separation and restoration, image deconvolution and thresholding, and medical image analysis \cite{UP1,UP2,UP3,UP4,UP5}.
The quaternion algebra
offers a simple and profound representation of signals wherein several components are to be controlled simultaneously. The development
of integral transforms for quaternion valued signals has found numerous applications in 3D computer graphics, aerospace engineering, artificial intelligence and colour image processing.
Due to the non-commutativity of quaternion multiplication, different
types of integral transforms have been generalized to quaternion algebra including Fourier and wavelet transforms \cite{UP6,UP7}, windowed  transform \cite{uncertainty1,uncertainty2}, etc. Therefore, the main objective of this article is to introduce the concept of the quaternionic shearlet transform and investigate its different properties using the machinery of quaternion Fourier transforms and  convolution. Moreover, we drive some uncertainty principle for the multivariate continuous quaternion shearlet transforms.
\ \\ \ \\
For a quaternion function $f \in L^{2}(\mathbb{R}^{2d},\mathbb{H})$ and a non zero quaternion function $g \in L^{2}(\mathbb{R}^{2d},\mathbb{H})$ called a quaternion shearlet .
The aim of this paper is to  generalize the continuous quaternion shearlet transform on $\mathbb{R}^{2}$ to $\mathbb{R}^{2d}$, called the multivariate two sided continuous quaternion shearlet transform which has been started in \cite{article3}.\\
Our purpose in this work  is to prove the Lieb uncertainty principle for the multivariate continuous quaternion shearlet transform. We also prove  the Logarithmic uncertainty principle . Last we study the Beckner uncertainty principle in terme of entropy  for the multivariate two sided continuous quaternion shearlet transform.
Our paper is organized as follows: In section 2, we present basic notions and notations related to the quaternion Fourier transform. In section 3, we recall the definition and we provide some results for the two sided quaternion shearlet transform useful in the sequel. In section 4, we provide  some uncertaintly principles for the two-sided multivariate quaternion shearlet transform.

\section{Generalities:}
In this section, we recall some basic definitions and properties of the Quaternion Fourier transform. For more details, see \cite{pitt,quaternion,MY1,MY2}.\\
The quaternion algebra was formally introduced by the Irish mathematician W.R Hamilton in 1843, and it is a generalization of complex numbers. The quaternion algebra over $\mathbb{R}$, denoted by $\mathbb{H}$, is an associative non-commutative four-dimensional algebra,
\begin{center}
$\mathbb{H} = \{q = q_{r} + iq_{i} + jq_{j} + kq_{k} \hspace{0.1 cm}|\hspace{0.1 cm} q_{r}, q_{i}, q_{j}, q_{k} \in \mathbb{R}\},$
\end{center}
which obey Hamilton's multiplication rules
\begin{center}
$ij = - ji = k$,\hspace{0.2 cm} $jk = -kj = i$,\hspace{0.2 cm}  $ki = -ik = j$,\hspace{0.2 cm} $i^{2} = j^{2} = k^{2} = ijk = -1$.
\end{center}
The quaternion conjugate of a quaternion $q$ is given by
\begin{center}
$\overline{q} = q_{r} - iq_{i} - jq_{j} - kq_{k}$, \hspace{0.3 cm} $ q_{r}, q_{i}, q_{j}, q_{k} \in \mathbb{R}$.
\end{center}
The quaternion conjugation is a linear anti-involution
\begin{center}
$\overline{pq} = \overline{q}\ \overline{p}$, \hspace{0.3 cm} $\overline{p+q} = \overline{p}+\overline{q}$, \hspace{0.3 cm} $\overline{\overline{p}} = p$.
\end{center}
The modulus of a quaternion $q$ is defined by
\begin{center}
$|q| = \sqrt{q\overline{q}} = \sqrt{q_{r}^{2}+ q_{i}^{2}+q_{j}^{2}+ q_{k}^{2}}$.
\end{center}
It is not difficult to see that
\begin{center}
$|pq| = |p| |q|$, \hspace{0.3 cm} $\forall p,q \in\mathbb{H}$.
\end{center}
The real scalar part $q_{r}$
\begin{center}
$q_{r} = Sc(q)$.
\end{center}
If $1\leq p < \infty$, the $L^{p}$-norm of $f:\mathbb{R}^{d}\longrightarrow\mathbb{H}$ is defined by
\begin{equation}\label{normp}
\|f\|_{p} = \bigg(\displaystyle\int_{\mathbb{R}^{d}} |f(x)|^{p} dx\bigg)^{\frac{1}{p}}.
\end{equation}
For $p = \infty$, $L^{\infty}(\mathbb{R}^{d},\mathbb{H})$ is a collection of essentially bounded measurable functions with the norm
$$\|f\|_{\infty} = \mbox{ess} \displaystyle\sup_{x\in \mathbb{R}^{d}} |f(x)|.$$
If $f \in L^{\infty}(\mathbb{R}^{d},\mathbb{H})$ is continuous then
\begin{equation}\label{norminfty}
\|f\|_{\infty} = \displaystyle\sup_{x\in \mathbb{R}^{d}} |f(x)|.
\end{equation}
For $p=2$, we can define the quaternion-valued inner product
\begin{equation}\label{&}
\langle f,g\rangle_{Q} = \displaystyle\int_{\mathbb{R}^{d}} f(x) \overline{g(x)} dx
\end{equation}
with symmetric real scalar part
\begin{equation}\label{pro}
(f,g)_{Q}\  = \ \dfrac{1}{2}\big[\langle f,g\rangle_{Q} + \langle g,f\rangle_{Q}\big]\ = \ \displaystyle\int_{\mathbb{R}^{d}} Sc(f(x) \overline{g(x)})\ dx = \ \ Sc\bigg(\displaystyle\int_{\mathbb{R}^{d}} f(x) \overline{g(x)}\ dx\bigg).
\end{equation}
Both (\ref{&}) and (\ref{pro}) lead to the $L^{2}(\mathbb{R}^{d},\mathbb{H})$-norm
\begin{equation}
\|f\|_{2,d} = \sqrt{(f,f)_{Q}} = \sqrt{\langle f,f\rangle_{Q}}  = \bigg(\displaystyle\int_{\mathbb{R}^{d}} |f(x)|^{2} dx\bigg)^{\frac{1}{2}}.
\end{equation}
As a consequence of the inner product (\ref{&}) we obtain the quaternion Cauchy-Schwartz inequality
\begin{equation}\label{cs}
\bigg|\displaystyle\int_{\mathbb{R}^{d}} f(x)\overline{g(x)} dx\bigg| \leq \bigg(\displaystyle\int_{\mathbb{R}^{d}} |f(x)|^{2}dx\bigg)^{\frac{1}{2}}\bigg(\displaystyle\int_{\mathbb{R}^{d}} |g(x)|^{2}dx\bigg)^{\frac{1}{2}}\hspace{0.2 cm} \forall f , g \in L^{2}(\mathbb{R}^{d},\mathbb{H}).
\end{equation}
\begin{definition} \cite{14}
The Quaternion Fourier transform of a function $f\in L^1(\mathbb{R}^{2d},\mathbb{H})\cap L^2(\mathbb{R}^{2d},\mathbb{H})$ is defined as
$$\mathcal{F}_{Q}(f)(u,v) := \int_{\mathbb{R}^{2d}}e^{-2i\pi u. x} f(x,y) e^{-2j\pi v. y} dx\ dy.$$
and it satisfies Plancherel's formula $\|\mathcal{F}_{Q}(f)\|_{2}=\|f\|_{2}.$ As a consequence $\mathcal{F}_{Q}$ extends to a
unitary operator on $L^2(\mathbb{R}^{2d},\mathbb{H})$ and satisfies Parseval's formula:
$$( f, g )_{Q} = ( \mathcal{F}_{Q}(f), \mathcal{F}_{Q}(g) )_{Q},\ \ \ \forall \ f,\ g\in L^2(\mathbb{R}^{2d},\mathbb{H}).$$
The inverse Quaternion Fourier transform of a function $f\in L^1(\mathbb{R}^{2d},\mathbb{H})$ is given as
$$\mathcal{F}_{Q}^{-1}( f )(x,y) = \mathcal{F}_{Q}( f )(-x,-y).$$
Thus, if $f \in L^1(\mathbb{R}^{2d},\mathbb{H})$ with $\mathcal{F}_{Q}(f)\in L^1(\mathbb{R}^{2d},\mathbb{H})$, then
$$f(x,y)=\int_{\mathbb{R}^{2d}}e^{2i\pi u. x} \mathcal{F}_{Q}(f)(u,v)\ e^{2j\pi v. y} \ du\ dv.$$
\end{definition}
\begin{definition}
The convolution of two functions $f,\ g \in L^1(\mathbb{R}^{2d},\mathbb{H})$ is the function $f*g$ defined by
$$(f * g)(x)=\int_{\mathbb{R}^{2d}}f(t)g(x-t)\ dt,\ \ \ x\in \mathbb{R}^{2d}.$$
\end{definition}
We have the following Theorem \ref{conditions}. With $\mathbb{R}^{2n}$ replacing $\mathbb{R}^{2}$ in \cite{article3} we get the results, we will not repeat the proof here.
\begin{theorem}\label{conditions}
Let $f$,\ $g$ be two quaternion functions and if we assume that
\begin{eqnarray}\label{hyp}
\mathcal{F}_{Q}(g)(u,v) e^{-2j\pi v.t} = e^{-2j\pi v.t} \mathcal{F}_{Q}(g)(u,v) \ \ \ and \ \ \ \mathcal{F}_{Q}(jg)(u,v) e^{-2j\pi v.t} =j  e^{-2j\pi v.t} \mathcal{F}_{Q}(g)(u,v); \ \ t \in \mathbb{R}^{d}
\end{eqnarray}
then the QFT of the convolution of $f\in L^{2}(\mathbb{R}^{2},\mathbb{H})$ and $g\in L^{2}(\mathbb{R}^{2},\mathbb{H})$ is given as
\begin{eqnarray}\label{conv}
\mathcal{F}_{Q}(f*g)(w) =  \mathcal{F}_{Q}(f)(w)\mathcal{F}_{Q}(g)(w)
\end{eqnarray}
\end{theorem}
\begin{remark}
\mbox{}
\begin{enumerate}
\item
Let $g(x,y) = g_{1}(x,y) + j g_{2}(x,y)$  where $g_{1}$ and $g_{2}$ in $L^{2}(\mathbb{R}^{2d},\mathbb{R})$ and $g(-x, y) = g(x, y)$ we have $g$ satisfies the condition (\ref{hyp}).
\item
 In \cite{bahri}, the authors gave a demonstration of the property by considering some other assumptions.
\end{enumerate}
\end{remark}
 \ \\
For all $f \in L^2(\mathbb{R}^{2d},\mathbb{H})$  and $g$ satisfies the condition of theorem \ref{conditions}, we have
$$f * g=\mathcal{F}_{Q}^{-1}( \mathcal{F}_{Q}(f)\cdot\mathcal{F}_{Q}(g) ).$$
Thus, if $f \in L^2(\mathbb{R}^{2d},\mathbb{H})$  and $g$ satisfies the condition of theorem \ref{conditions}, the function $f * g$ belongs to $L^2(\mathbb{R}^{2d},\mathbb{H})$ if and only if
$\mathcal{F}_{Q}(f)\cdot\mathcal{F}_{Q}(g)$ belongs to $L^2(\mathbb{R}^{2d},\mathbb{H})$, and in this case, we have
$$\mathcal{F}_{Q}(f * g)=\mathcal{F}_{Q}(f)\cdot\mathcal{F}_{Q}(g).$$
Then, for all $f \in L^2(\mathbb{R}^{2d},\mathbb{H})$  and $g$ satisfies the condition of theorem \ref{conditions}, we have
\begin{eqnarray}\label{theorem}
\int_{\mathbb{R}^{2d}}|f * g(x)|^2 d\ x=\int_{\mathbb{R}^{2d}}|\mathcal{F}_{Q}(f)(\xi)|^2|\mathcal{F}_{Q}(g)(\xi)|^2 d\xi,
\end{eqnarray}
where both sides are finite or both sides are infinite.
\section{Multivariate  continuous Quaternion Shearlet Transform:}
The Shearlet transform was originally developed in the inaugural paper \cite{lAU} and became a very useful mathematical tool, which has been widely used in characterization of function spaces as well as in signal and image processing, see \cite{dahlkesh1aa, dahlkesh, shguo1, shguo1&, kutla&, kutla&&, kutla, sliu}.\\
For more details on shearlet transform, the reader can see \cite{dahlkesh,fitouhi,shguo}.\\ \ \\
In this section, we introduce the Multivariate  continuous quaternion shearlet transform and we establish some new results (Parseval formula, inversion formula,etc.).\\
Let $I_{2n}$ denote the $(2n, 2n)$-identity matrix and $0_{2n}$, resp. $1_{2n}$ the vectors with $2n$ entries $0$, resp. $1$. For $a \in\mathbb{R}^* := \mathbb{R}\backslash \{0\}$ and $s\in\mathbb{R}^{2n-1}$, we set
$$
A_a :=\left(
   \begin{array}{ccc}
     a &  &0_{2n-1}^T\\
     \\
     0_{2n-1} &  &sg(a)|a|^{\frac{1}{2n}}\  I_{2n-1}
   \end{array}
 \right)\ \ \ \mbox{and}\ \ \ S_s :=\left(
   \begin{array}{ccc}
     1 &  &s^T\\
     \\
     0_{2n-1} &  &I_{2n-1}
   \end{array}
 \right),
$$
where $s^T=(s_1 ,..., s_{2n-1})$.
\begin{remark} Let $a,b \in\mathbb{R}^*$ and $s\in\mathbb{R}^{2n-1}$. Then, we have
\begin{enumerate}
\item $A_{a}^{-1}=A_{a^{-1}}$,\ \ $A_{a}^{T}=A_{a}$, \ $A_{a}A_{b}=A_{ab}$ \ \mbox{and}\ $|\det{A_{a}}| = |a|^{2-\frac{1}{2n}}$.
\item \begin{equation}\label{sh1}
S_s^{-1}=\left(
   \begin{array}{ccc}
     1 &  &-s^T\\
     \\
     0_{2n-1} &  &I_{2n-1}
   \end{array}
 \right) \ \ \mbox{and}\ \ \ S_{s}^{T}=\left(
   \begin{array}{ccc}
     1 &  &0_{2n-1}\\
     \\
     s &  &I_{2n-1}
   \end{array}
 \right).
 \end{equation}
 \item \begin{equation}S_{s}A_{a}=\left(
   \begin{array}{ccccc}
     a & sg(a)|a|^{\frac{1}{2n}}s_1& sg(a)|a|^{\frac{1}{2n}}s_2 &\ldots &sg(a)|a|^{\frac{1}{2n}}s_{2n-1}\\
     & & & &\\
      0 & sg(a)|a|^{\frac{1}{2n}}&0&\ldots &\ldots\\
       & & & &\\
      \vdots & \ddots&\ddots&\ddots &\vdots\\
       & & & &\\
      \vdots & \ddots& \ddots&\ddots&\ddots\\
       & & & &\\
     \vdots& \ldots & \ddots&\ddots&\ddots
   \end{array}
 \right).\end{equation}
  \item
\begin{equation}
 |\det{S_{s}A_{a}}| = |a|^{2-\frac{1}{2n}}.
\end{equation}
 \item
 \begin{equation}\label{rechang}
 \displaystyle A_{a}^{T}\ S_{s}^{T} \lambda=\Big(a\lambda_1 ,\  sg(a)|a|^{\frac{1}{2n}}(s_{1}\lambda_{1}+\lambda_{2}), ...,\  sg(a)|a|^{\frac{1}{2n}}(s_{2n-1}\lambda_{1}+\lambda_{2n})\Big),
 \end{equation} for every $\lambda=(\lambda_{1},..., \lambda_{2n})^{T}\in\mathbb{R}^{2n}.$

\end{enumerate}
\end{remark}
The set $\mathbb{R}^*\times\mathbb{R}^{2n-1}\times\mathbb{R}^{2n}$ endowed with the operation
 \begin{equation}
 (a, s, t)\circ(a', s', t')= (aa',\ s+|a|^{1-\frac{1}{2n}}s',\ t+S_{s}A_{a}t'), \end{equation}
 is a locally compact group $\mathcal{S}$ which we call full Shearlet group. The left and right Haar measures on $\mathcal{S}$ are given by
$$d\mu_{l}(a, s, t)=\frac{1}{|a|^{2n+1}}da\ ds\ dt\ \ \ \mbox{and}\ \ \ d\mu_{r}(a, s, t)=\frac{1}{|a|}da\ ds\ dt.$$
We defined by $\mu_{l}$ the normalized Lebesgue measure on $\mathbb{R}^{*}\times\mathbb{R}^{2d-1}\times\mathbb{R}^{2d}$ .

Notice that the Shearlet group is isomorphic to the locally compact group $\mathcal{G}\times\mathbb{R}^{2n}$, where
 $$\mathcal{G}=\{S_{s}A_{a} : a\in \mathbb{R}^*,\ s\in\mathbb{R}^{2n-1}\}.$$
 \begin{remark} For all $(a,s,t),(a',s',t')\in \mathbb{R}^*\times\mathbb{R}^{2n-1}\times\mathbb{R}^{2n}$, we have
\begin{equation}\label{ssaat} S_{s}A_{a}S_{s'}A_{a'} = S_{s+|a|^{1-\frac{1}{2n}}s'}A_{aa'}.\end{equation}
In particular, if $s'=-|a|^{\frac{1}{2n}-1}s$ and $a'=a^{-1}$, then
\begin{equation}\label{ssaat*}S_{s}A_{a}S_{-|a|^{\frac{1}{2n}-1}s}A_{a^{-1}}=I_{2n}.
\end{equation}
By the previous relations it follows that $e := (1, 0_{2n-1}, 0_{2n})$ is the neutral element in $\mathcal{S}$ and
that the inverse of $(a, s, t)$ is given by
$$(a, s, t)^{-1}=(a^{-1},\ -|a|^{\frac{1}{2n}-1}s,\ -A_{a}^{-1}S_{s}^{-1}t).$$
 \end{remark}
In the following, we use only the left Haar measure and use the abbreviation $d\mu=d\mu_{l}$. We denote by $L^p(d\mu)$,\ $p\in[1,+\infty]$, the Lebesgue space formed by the measurable functions $f$ on $\mathbb{R}^*\times\mathbb{R}^{2n-1}\times\mathbb{R}^{2n}$ such that $\|f\|_{p,\mu}<\infty$, with

\begin{eqnarray*}\|f\|_{p,\mu}=\left \{
\begin{aligned}&\left(\int_{\mathbb{R}}\int_{\mathbb{R}^{2n-1}}\int_{\mathbb{R}^{2n}}|f(a,s,t)|^{p}d\mu(a,s,t)\right)^{\frac{1}{p}},
&\mbox{if}& &1\leq p <+\infty\\
\\
&\sup_{(a,s,t)\in \mathbb{R}^*\times \mathbb{R}^{2n-1}\times\mathbb{R}^{2n}}|f(a,s,t)| ,&\mbox{if}&  &p =+\infty
\end{aligned}
  \right..
\end{eqnarray*}
Let $f$ and $g \in L^{2}(d\mu)$, then the inner product in $L^{2}(d\mu)$ is given by
  \begin{center}
  $\langle f,g\rangle _{\mu} = \displaystyle\int_{\mathbb{R}}\int_{\mathbb{R}^{2n-1}}\int_{\mathbb{R}^{2n}} Sc\big( f(a,s,t)\overline{g(a,s,t)}\big) d\mu(a,s,t) = \displaystyle\int_{\mathbb{R}}\int_{\mathbb{R}^{2n-1}} (f(a,s,.),g(a,s,.))_{Q} \dfrac{da}{|a|^{2n+1}} \ ds$.
  \end{center}
\begin{definition}\label{defsh} Let $a \in\mathbb{R}^* ,\ s\in\mathbb{R}^{2n-1}$ and $t\in\mathbb{R}^{2n}$.
 The dilation operator $\pi_{a,s,t}$ is defined for every function $\psi \in L^2(\mathbb{R}^{2n},\mathbb{H})$ by
 \begin{equation}\label{psiast}
 \pi_{a, s, t}(\psi)(x) = \psi_{a,s,t}(x) := |a|^{\frac{1}{4n}-1}\psi(A_{a}^{-1}S_{s}^{-1}(x - t)).\end{equation}
 These operators satisfy:
 \begin{enumerate}
\item For all $(a,s,t),(a',s',t')\in \mathbb{R}^*\times\mathbb{R}^{2n-1}\times\mathbb{R}^{2n}$, we have
\begin{equation}
 \pi_{a, s, t}\circ \pi_{a', s', t'}= \pi_{aa',\ s+|a|^{1-\frac{1}{2n}}s',\ t+S_{s}A_{a}t'} .
 \end{equation}
 \item For every $\psi \in L^2(\mathbb{R}^{2n},\mathbb{H})$,
 \begin{equation}
 \pi_{1, 0_{2n-1}, 0_{2n}}(\psi)= \psi.
 \end{equation}
\end{enumerate}
 \end{definition}
\begin{proposition}
\mbox{}\\
\begin{enumerate}
\item
Let $(a,s,t)\in \mathbb{R}^*\times\mathbb{R}^{2n-1}\times\mathbb{R}^{2n}$, then for all $\varphi, \psi \in L^2(\mathbb{R}^{2n},\mathbb{H})$, we have\\
    \ \\
  (i)
$$\displaystyle
 \int_{\mathbb{R}^{2n}}\psi_{a, s, t}(x)\varphi(x)\ dx = \int_{\mathbb{R}^{2n}}\psi(x)\varphi_{a^{-1}, -s|a|^{\frac{1}{2n}-1}, -A_{a}^{-1}S_{s}^{-1}t}(x)\ dx.$$
(ii)
$$\mathcal{F}_{Q}(\psi_{a,s,t})(\lambda) = |a|^{1-\frac{1}{4n}} \bigg(\displaystyle\prod_{p=1}^{n}e^{-2i\pi t_{p}\lambda_{p}} \bigg)  \mathcal{F}_{Q}(\psi)(A_{a}^{T}S_{s}^{T}\lambda)\bigg(\displaystyle\prod_{q=n+1}^{2n}e^{-2j\pi t_{q}\lambda_{q}} \bigg).$$
 \item Let $p\geq 1$. For every $\psi \in L^p(\mathbb{R}^{2n},\mathbb{H})$ and $(a,s,t)\in \mathbb{R}^*\times\mathbb{R}^{2n-1}\times\mathbb{R}^{2n}$, the function $\psi_{a, s, t}$ belongs to $L^p(\mathbb{R}^{2n},\mathbb{H})$ and we have
\begin{eqnarray}\label{psino}
\|\psi_{a, s, t}\|_p= |a|^{(\frac{1}{4n}-1)\frac{p-2}{p}}\|\psi\|_p.
\end{eqnarray}
In particular, $\pi_{a, s, t}$ is an isometric isomorphism from $L^2(\mathbb{R}^{2n},\mathbb{H})$ onto itself whose inverse operator is $\pi_{a^{-1}, -s|a|^{\frac{1}{2n}-1}, -A_{a}^{-1}S_{s}^{-1}t}$.
\end{enumerate}
\end{proposition}
\begin{definition} A function $\psi$ in $ L^2(\mathbb{R}^{2n},\mathbb{H})$ is admissible quaternion shearlet if and only if it fulfills the admissibility condition
\begin{eqnarray}
C_{\psi}=\int_{\mathbb{R}^{2n}}\frac{|\mathcal{F}_{Q}(\psi^{*})(\lambda_{1},\tilde{\lambda})|^2}{|\lambda_1|^{2n}} d\lambda <\infty.
\end{eqnarray}
where $\tilde{\lambda}= (\lambda_{2},\lambda_{3},...,\lambda_{2n}),$ and $\psi^{*}(x) = \check{\overline{\psi(x)}} =  \overline{\psi(-x)}$.
\end{definition}
\begin{lemma} Let $\psi$ be an admissible quaternion shearlet in $L^2(\mathbb{R}^{2n}, \mathbb{H})$. Then
\begin{eqnarray}\label{cpsish}
C_{\psi}=\int_{\mathbb{R}}\int_{\mathbb{R}^{2n-1}}\frac{\Big|\mathcal{F}_{Q}(\psi^{*})(A_{a}^{T}S_{s}^{T}\lambda)\Big|^2}{|a|^{\frac{4n^{2}-2n+1}{2n}}}da\ ds,
\end{eqnarray}
is satisfied.
\end{lemma}
\begin{proof} Let $\lambda\in \mathbb{R}^{2n}$. From the relation $\ref{rechang}$, we have \\

 $\displaystyle\int_{\mathbb{R}}\int_{\mathbb{R}^{2n-1}}\frac{\Big|\mathcal{F}_{Q}(\psi^{*})(A_{a}^{T}S_{s}^{T}\lambda)\Big|^2}
{|a|^{\frac{4n^2-2n+1}{2n}}}da \ ds$
\begin{eqnarray*}
&=&\int_{\mathbb{R}}\int_{\mathbb{R}^{2n-1}}\frac{|\mathcal{F}_{Q}(\psi^{*})\Big((a\lambda_1 ,\  sg(a)|a|^{\frac{1}{2n}}(s_{1}\lambda_{1}+\lambda_{2}), ...,\  sg(a)|a|^{\frac{1}{2n}}(s_{2n-1}\lambda_{1}+\lambda_{2n})\Big)|^2}
{|a|^{\frac{4n^2-2n+1}{2n}}}da\ ds.
\end{eqnarray*}
Substituting $w_1 =a\lambda_1$, i.e, $d w_1 =|\lambda_1|\ da$. Next, we put
$$\begin{array}{ccc}
     w_2=& sg(\frac{w_1}{\lambda_1})|\frac{w_1}{\lambda_1}|^{\frac{1}{2n}}(s_{1}\lambda_{1}+\lambda_{2})&\\
     .& &\\
     .& &\\
     .& &\\
    w_n=& sg(\frac{w_1}{\lambda_1})|\frac{w_1}{\lambda_1}|^{\frac{1}{2n}}(s_{2n-1}\lambda_{1}+\lambda_{2n}),&
   \end{array}$$
i.e, $dw_2= |\frac{w_1}{\lambda_1}|^{\frac{1}{2n}}|\lambda_{1}|\ ds_{1}$,\ ...\ , $dw_{2n}= |\frac{w_1}{\lambda_1}|^{\frac{1}{2n}}|\lambda_{1}|\ ds_{2n-1}$.\\
Then
\begin{eqnarray*}
\int_{\mathbb{R}}\int_{\mathbb{R}^{2n-1}}\frac{\Big|\mathcal{F}_{Q}(\psi^{*})(A_{a}^{T}S_{s}^{T}\lambda)\Big|^2}
{|a|^{\frac{4n^2-2n+1}{2n}}}da \ ds&=&\int_{\mathbb{R}^{2n}}\frac{|\mathcal{F}_{Q}(\psi^{*})\Big((w_1 ,\  w_2, ...,\  w_{2n})\Big)|^2}
{|\frac{w_1}{\lambda_1}|^{\frac{4n^2-2n+1}{2n}}}\frac{dw_{1}}{|\lambda_{1}|}\frac{dw_{2}...dw_{2n}}
{(|\frac{w_1}{\lambda_1}|^{\frac{1}{2n}}|\lambda_{1}|)^{2n-1}}\\
&=&\int_{\mathbb{R}^{2n}}
\frac{\Big|\mathcal{F}_{Q}(\psi^{*})(w_1 , \widetilde{w})\Big|^2}
{|w_1|^{2n}} dw_1 \ d\widetilde{w}\\
&=& C_{\psi},
\end{eqnarray*}
where $\widetilde{w}=(w_2 ,...,w_{2n})$.
\end{proof}
\begin{definition}\label{desha}
Let $\psi$ be an  admissible quaternion shearlet in $L^2(\mathbb{R}^{2n},\mathbb{H})$. The continuous
quaternion shearlet transform $\mathcal{SH}_{\psi}^{Q}$ is defined for every quaternion function $f$ in $L^2(\mathbb{R}^{2n},\mathbb{H})$, by
$$\mathcal{SH}_{\psi}^{Q}(f)(a,s,t) = \langle f, \psi_{a,s,t}  \rangle_{Q} = \displaystyle\int_{\mathbb{R}^{2n} } f(x)\overline{\psi_{a,s,t}(x)}\ dx. $$
\end{definition}
\begin{proposition} \label{123}
Let $\psi$ be an admissible quaternion  shearlet in $ L^2(\mathbb{R}^{2n},\mathbb{H})$. For every $f\in L^2(\mathbb{R}^{2n},\mathbb{H})$, the continuous quaternion shearlet transform can be represented as
\begin{equation}
\mathcal{SH}^{Q}_{\psi}(f)(a,s,t) =  \big(f*\psi_{a,s,0}^{*}\big)(t),
\end{equation}
where $\psi_{a,s,0}^{*}(x) = \overline{\psi_{a,s,0}(-x)}$.
\end{proposition}
\begin{proof}
We observe that
\begin{eqnarray*}
\mathcal{SH}_{\psi}^{Q}(f)(a,s,t) & =& \displaystyle\int_{\mathbb{R}^{2n}} f(x) \overline{\psi_{a,s,t}(x)}\ dx \\
& = & |a|^{\frac{1}{4n}-1} \displaystyle\int_{\mathbb{R}^{2n}} f(x)\overline{ \psi(A_{a}^{-1}S_{s}^{-1}(x-t))}\  dx\\
& = & |a|^{\frac{1}{4n}-1} \displaystyle\int_{\mathbb{R}^{2n}} f(x)\overline{ \psi(A_{a}^{-1}S_{s}^{-1}(-(t-x))}\  dx\\
&  =  & \displaystyle\int_{\mathbb{R}^{2n}} f(x)\psi_{a,s,0}^{*}(t-x)\ dx\\
&  =   & \big(f*\psi_{a,s,0}^{*}\big)(t).
\end{eqnarray*}
\end{proof}
\begin{theorem}
Let $\psi$ be an admissible quaternion shearlet in $L^{2}(\mathbb{R}^{2n},\mathbb{H})$. For every function $f$ in $L^{2}(\mathbb{R}^{2n},\mathbb{H})$, the quaternion shearlet transform satisfies  the following properties
\begin{itemize}
\item[(i)] \textbf{Linearity:}
\begin{eqnarray*}
\mathcal{SH}_{\psi}^{Q}(\alpha f + \beta g)(a,s,t) = \alpha\ \mathcal{SH}_{\psi}^{Q}(f)(a,s,t) + \beta\ \mathcal{SH}_{\psi}^{Q}(g)(a,s,t),\ \ \ \alpha,\beta \in \mathbb{H}.
\end{eqnarray*}
\item[(ii)] \textbf{Anti-linearity:}
\begin{eqnarray*}
\mathcal{SH}_{\alpha\psi+\beta\phi}^{Q}(f)(a,s,t) = \mathcal{SH}_{\psi}^{Q}(f)(a,s,t) \overline{\alpha} + \mathcal{SH}_{\phi}^{Q}(f)(a,s,t) \overline{\beta},\ \ \ \alpha,\beta \in \mathbb{H}.
\end{eqnarray*}
\item[(iii)] \textbf{Translation:}
\begin{eqnarray*}
\mathcal{SH}_{\psi}^{Q}(T_{\alpha}f)(a,s,t) = \mathcal{SH}_{\psi}^{Q}(f)(a,s,t-\alpha),\ \ \alpha\in\mathbb{R}^{2n},
\end{eqnarray*}
where $T_{\alpha}f(x) = f(x-\alpha)$.

\end{itemize}
\end{theorem}
\begin{proof}
\begin{itemize}
\item[(i)] Let $\alpha$, $\beta$ in $\mathbb{H}$. Then
$$\begin{tabular}{lll}
$\mathcal{SH}_{\psi}^{Q}(\alpha f + \beta g)(a,s,t)$  & $=$ & $\displaystyle\int_{\mathbb{R}^{2n}} (\alpha f(x) + \beta g(x))\overline{\psi_{a,s,t}(x)} dx$\\
& $=$ & $\alpha \displaystyle\int_{\mathbb{R}^{2n}}  f(x)\overline{\psi_{a,s,t}(x)} dx + \beta \displaystyle\int_{\mathbb{R}^{2n}}  g(x)\overline{\psi_{a,s,t}(x)} dx$\\
& $=$ & $ \alpha\ \mathcal{SH}_{\psi}^{Q}(a,s,t) + \beta\ \mathcal{SH}_{\psi}^{Q}g(a,s,t).$
\end{tabular}$$
\item[(ii)] Let $\alpha$, $\beta$ in $\mathbb{H}$. Then
$$ \begin{tabular}{lll}
$\mathcal{SH}_{\alpha\psi+\beta\phi}^{Q}(f)(a,s,t)$ & $=$ & $\displaystyle\int_{\mathbb{R}^{2n}} f(x) \overline{(\alpha\psi_{a,s,t}(x) + \beta\phi_{a,s,t}(x)) } dx$\\
& $=$ & $\displaystyle\int_{\mathbb{R}^{2n}} f(x) \big( \overline{\psi_{a,s,t}(x)}\ \overline{\alpha} + \overline{\phi_{a,s,t}(x)}\ \overline{\beta}\big)  dx$\\
& $=$ & $\mathcal{SH}_{\psi}^{Q}(f)(a,s,t)\ \overline{\alpha} + \mathcal{SH}_{\phi}^{Q}(f)(a,s,t)\ \overline{\beta}$.
\end{tabular}$$
\item[(iii)] Let $\alpha$  in $\mathbb{H}$. Then
$$\begin{tabular}{lllll}
$\mathcal{SH}_{\psi}^{Q}(T_{\alpha}f)(a,s,t)$
 & $=$ & $\displaystyle\int_{\mathbb{R}^{2n}} T_{\alpha}f(x)\overline{ \psi_{a,s,t}(x)}\ dx$\\
 & $=$ & $\displaystyle\int_{\mathbb{R}^{2n}} f(x-\alpha)\overline{ \psi_{a,s,t}(x)}\ dx$\\
 & $=$ & $|a|^{\frac{1}{4n}-1} \displaystyle\int_{\mathbb{R}^{2n}}  f(x-\alpha)  \overline{\psi(A_{a}^{-1}S_{s}^{-1}(x-t))}\ dx$\\
 & $=$ & $|a|^{\frac{1}{4n}-1} \displaystyle\int_{\mathbb{R}^{2n}}  f(z)  \overline{\psi(A_{a}^{-1}S_{s}^{-1}(z-(t-\alpha))}\ dz$\\
 & $=$ & $\displaystyle\int_{\mathbb{R}^{2n}}  f(z) \overline{\psi_{a,s,t-\alpha}(z)}\ dz$\\
 & $=$ & $\mathcal{SH}_{\psi}^{Q}(f)(a,s,t-\alpha) .$
\end{tabular}$$
 \end{itemize}
\end{proof}
\begin{lemma}\label{l&}
Assume that  $\psi$ satisfies the assumption of theorem \ref{conditions} then for every $f\in L^{2}(\mathbb{R}^{2n},\mathbb{H})$, we have
\begin{equation}\label{fousht}
\begin{tabular}{lll}
$\mathcal{F}_{Q}(\mathcal{SH}_{\psi}^{Q}(f)(a,s,.))(w)$
& $=$ & $|a|^{1-\frac{1}{4n}} \mathcal{F}_{Q}(f)(w) \mathcal{F}_{Q}(\psi^{*})(A_{a}^{T}S_{s}^{T}w)$ ,\ \ $w \in \mathbb{R}^{2n}$ \\
& $=$ & $ |a|^{1-\frac{1}{4n}} \mathcal{F}_{Q}(f)(w) \mathcal{F}_{Q}(\psi^{*})(aw_{1},sg(a)|a|^{\frac{1}{2n}} (sw_{1}+ \tilde{w})),$
\end{tabular}
\end{equation}
where $\tilde{w} = (w_{2},w_{3},...,w_{2n}).$
\end{lemma}
\begin{proof}By invoking the convolution Theorem for QFT we have
\begin{eqnarray*}
\mathcal{F}_{Q}(\mathcal{SH}_{\psi}^{Q}(f)(a,s,.))(w)
& = & \mathcal{F}_{Q}(f*\psi_{a,s,0}^{*}(.))(w)\\
& = & \mathcal{F}_{Q}(f)(w) \mathcal{F}_{Q}(\psi_{a,s,0}^{*})(w)\\
& = & |a|^{1-\frac{1}{4n}} \mathcal{F}_{Q}(f)(w) \mathcal{F}_{Q}(\psi^{*})(A_{a}^{T}S_{s}^{T}w).
\end{eqnarray*}
This achieves the proof.
\end{proof}
\begin{example}
We choose
$$\psi(x) =\psi(x_1,...,x_{2n})=
\left\{
\begin{array}{ccc}
\displaystyle\prod_{p=1}^{2n}e^{-x_{p}}, &  \ \ \ & x_{p} \in [0,+\infty[\\
0,  &   \ \ \  & \mbox{elsewhere}
\end{array}
\right.$$
and we consider a quaternion valued function $f$ defined by
\begin{eqnarray*}
f(x) = e^{-\frac{|x|^{2}}{\gamma^{2}}} + j e^{-|x|^{2}} , \ \ \ \gamma \in \mathbb{R}.
\end{eqnarray*}
The Quaternion Fourier transform with respect to the function $f$ is obtained as follows
\begin{eqnarray*}
\mathcal{F}_{Q}(f)(w) & = &
 (\pi^{n}\gamma^{2n})\  e^{-\pi^{2}\gamma^{2}|w|^{2}} + j\pi^{n} e^{-\pi^{2}|w|^{2}}
\end{eqnarray*}
similary $\mathcal{F}_{Q}(\psi^{*})(aw_{1},sg(a)|a|^{\frac{1}{2}} (sw_{1}+w_{2}))$ can be obtained as \ \\ \ \\
$\mathcal{F}_{Q}(\psi^{*})(aw_{1},sg(a)|a|^{\frac{1}{2n}} (sw_{1}+\tilde{w})) $
\begin{eqnarray*}
& = & \bigg(\dfrac{2i\pi w_{1}+1}{4\pi^{2}w_{1}^{2}+1}\bigg)
\displaystyle\prod_{p=2}^{n}\bigg(\dfrac{1 + 2i\pi sg(a)|a|^{\frac{1}{2n}}(sw_{1}+w_{p})}{1 + 4\pi^{2}|a|^{\frac{1}{n}}(sw_{1}+w_{p})^{2}}\bigg)
\displaystyle\prod_{q=n+1}^{2n}\bigg(\dfrac{1 + 2j\pi sg(a)|a|^{\frac{1}{2n}}(sw_{1}+w_{q})}{1 + 4\pi^{2}|a|^{\frac{1}{n}}(sw_{1}+w_{q})^{2}}\bigg)
\end{eqnarray*}
By virtue to lemma \ref{l&} , we have\ \\ \ \\
$\mathcal{F}_{Q}(\mathcal{SH}^{Q}_{\psi}(f))(a,s,.))(w) = |a|^{1 - \frac{1}{4n}}\ (\pi^{n}\gamma^{2n})\ [ e^{-\pi^{2}\gamma^{2}|w|^{2}} + j\pi^{n} e^{-\pi^{2}|w|^{2}}] \bigg(\dfrac{2i\pi w_{1}+1}{4\pi^{2}w_{1}^{2}+1}\bigg)$
\begin{eqnarray*}
\times\displaystyle\prod_{p=2}^{n}\bigg(\dfrac{1 + 2i\pi sg(a)|a|^{\frac{1}{2n}}(sw_{1}+w_{p})}{1 + 4\pi^{2}|a|^{\frac{1}{n}}(sw_{1}+w_{p})^{2}}\bigg)
\displaystyle\prod_{q=n+1}^{2n}\bigg(\dfrac{1 + 2j\pi sg(a)|a|^{\frac{1}{2n}}(sw_{1}+w_{q})}{1 + 4\pi^{2}|a|^{\frac{1}{n}}(sw_{1}+w_{q})^{2}}\bigg).
\end{eqnarray*}
\end{example}

\begin{lemma}\label{1&&}
Let $\psi$ be an admissible quaternion shearlet  satisfying the assumption of Theorem \ref{conditions} . Then for every function $f$ in $L^2(\mathbb{R}^{2n},\mathbb{H})$, we have
\begin{eqnarray}\int_{\mathbb{R}^{2n}}|\mathcal{F}_{Q}(\mathcal{SH}_{\psi}^{Q}(f)(a,s,.))(\lambda)|^{2} d\lambda =
\int_{\mathbb{R}^{2n}}|\mathcal{SH}_{\psi}^{Q}(f)(a,s,t)|^{2} d t.
\end{eqnarray}
\end{lemma}
\begin{proof} Using Lemma $\ref{l&}$ and equation (\ref{theorem}) , we get
\begin{eqnarray*}
\int_{\mathbb{R}^{2n}}|\mathcal{SH}_{\psi}^{Q}(f)(a,s,t)|^{2} d t & = & \int_{\mathbb{R}^{2n}}|f * \psi_{a,s,0}^{\ast}(t)|^{2} d t\\
&=&\int_{\mathbb{R}^{2n}}|\mathcal{F}_{Q}(f)(\lambda)|^2|\mathcal{F}_{Q}(\psi^{\ast}_{a,s,0})(\lambda)|^2 d\lambda\\
& = & \int_{\mathbb{R}^{2n}}|\mathcal{F}_Q (\mathcal{SH}_{\psi}^{Q}(f)(a,s,.))(\lambda)|^{2} d \lambda .
\end{eqnarray*}
\end{proof}
\begin{theorem}\label{thplan} Let $\psi$ be an admissible quaternion shearlet in $L^{2}(\mathbb{R}^{2n},\mathbb{H})$ satisfying the assumption of theorem \ref{conditions}. Then, we have
\begin{enumerate}

\item(Parseval's theorem for $\mathcal{SH}_{\psi}^{Q}$) For all functions $f$ and $g$ in $L^2(\mathbb{R}^{2n},\mathbb{H})$,
\begin{eqnarray}\label{pstsh}
\langle\mathcal{SH}_{\psi}^{Q}(f), \mathcal{SH}_{\psi}^{Q}(g)\rangle_{\mu}=C_{\psi}( f,g )_{Q},
\end{eqnarray}
where $\langle . \rangle_{\mu}$ is the inner product  on $L^2(d\mu)$.
\item (Plancherel's theorem for $\mathcal{SH}_{\psi}^{Q}$) For every function $f$ in $L^2(\mathbb{R}^{2n},\mathbb{H})$,
\begin{eqnarray}\label{Ptsh}
\int_{\mathbb{R}}\int_{\mathbb{R}^{2n-1}} \int_{\mathbb{R}^{2n}}|\mathcal{SH}_{\psi}^{Q}(f)(a,s,t)|^{2}\ d\mu(a,s,t) = C_{\psi}\|f\|_{2}^{2}.
\end{eqnarray}
\end{enumerate}
\end{theorem}
\begin{proof}
\begin{enumerate}
\item
Using Fubini's theorem , lemma \ref{l&} and applying Parseval's formula for the two sided QFT to the $t$-integral into the left side of the Eq. (\ref{pstsh}) yields\\\\
$Sc\bigg(\displaystyle\int_{\mathbb{R}}\int_{\mathbb{R}^{2n-1}}\int_{\mathbb{R}^{2n}} \mathcal{S}_{\psi}^{Q}(f)(a,s,t)\overline{\mathcal{S}_{\psi}^{Q}(g)(a,s,t)} d\mu(a,s,t)\bigg)$\\
\begin{eqnarray*}
& = & \displaystyle\int_{\mathbb{R}}\int_{\mathbb{R}^{2n-1}}Sc\bigg(\int_{\mathbb{R}^{2n}} \mathcal{S}_{\psi}^{Q}(f)(a,s,t)\overline{\mathcal{S}_{\psi}^{Q}(g)(a,s,t)}\  dt \bigg) \dfrac{da}{|a|^{2n+1}}\ ds\\
& = & \displaystyle\int_{\mathbb{R}}\int_{\mathbb{R}^{2n-1}} Sc\bigg(\int_{\mathbb{R}^{2n}} \mathcal{F}_{Q}( \mathcal{S}_{\psi}^{Q}(f)(a,s,.))(w)\overline{\mathcal{F}_{Q}(\mathcal{S}_{\psi}^{Q}(g)(a,s,.))(w)} \  dw \bigg) \dfrac{da}{|a|^{2n+1}}\ ds\\
& = & \displaystyle\int_{\mathbb{R}}\int_{\mathbb{R}^{2n-1}} Sc\bigg(\int_{\mathbb{R}^{2n}} \mathcal{F}_{Q}(f)(w)\big|\mathcal{F}_{Q}(\psi^{*})(A_{a}^{T}S_{s}^{T}w)\big|^{2} \overline{\mathcal{F}_{Q}(g)(w)}\ dw \bigg) \dfrac{da}{|a|^{\frac{4n^{2}-2n+1}{2n}}}\ ds\\
& = & Sc\bigg(\int_{\mathbb{R}^{2}} \mathcal{F}_{Q}(f)(w) \bigg(\displaystyle\int_{\mathbb{R}}\int_{\mathbb{R}} \big|\mathcal{F}_{Q}(\psi^{*})(A_{a}^{T}S_{s}^{T}w)\big|^{2}\dfrac{da}{|a|^{\frac{4n^{2}-2n+1}{2n}}}\ ds \bigg) \overline{\mathcal{F}_{Q}(g)(w)}\ dw \bigg) \\
& = & C_{\psi} ( \mathcal{F}_{Q}(f),\mathcal{F}_{Q}(g)) _{Q}\\
& = & C_{\psi} ( f,g )_{Q}.
\end{eqnarray*}
\item
The result follows by taking $f = g$ in Eq.\ (\ref{pstsh}).
\end{enumerate}
\end{proof}
\begin{theorem} (Reconstruction formula for $\mathcal{SH}_{\psi}^{Q}$)\\
Let $\psi$ be an admissible quaternion shearlet in $L^{2}(\mathbb{R}^{2n},\mathbb{H})$ satisfying the assumption of theorem \ref{conditions}, such that $|\psi|$ is an admissible quaternion shearlet (in particular for $\psi$ non-negative) in $L^2(\mathbb{R}^{2n},\mathbb{H})$, for every function $f$ in $L^2(\mathbb{R}^{2n},\mathbb{H})$, we have
\begin{eqnarray}
f(x)=\frac{1}{C_{\psi}}\int_{\mathbb{R}}\int_{\mathbb{R}^{2n-1}} \int_{\mathbb{R}^{2n}}\mathcal{SH}_{\psi}(f)(a,s,t) \psi_{a,s,t}(x)\ d\mu(a,s,t)
\end{eqnarray}
weakly in $L^2(\mathbb{R}^{2n},\mathbb{H})$.
\end{theorem}
\begin{proof} Let $f,g\in L^2(\mathbb{R}^{2n},\mathbb{H})$. From the relations $(\ref{Ptsh})$ and $(\ref{pstsh})$, we have

\begin{eqnarray*}
C_{\psi}( f,g )_{Q} &=& \int_{\mathbb{R}}\int_{\mathbb{R}^{2n-1}} \int_{\mathbb{R}^{2n}} Sc\big(\mathcal{SH}_{\psi}^{Q}(f)(a,s,t) \overline{\mathcal{SH}_{\psi}^{Q}(g)(a,s,t)}\big)\ d\mu(a,s,t)\\
&=&\int_{\mathbb{R}}\int_{\mathbb{R}^{2n-1}}\int_{\mathbb{R}^{2n}}Sc\bigg(\mathcal{SH}_{\psi}^{Q}(f)(a,s,t)\overline{\bigg(\int_{\mathbb{R}^{2n}}g(x)\overline{ \psi_{a,s,t}(x)}\ dx\bigg)}\bigg) d\mu(a,s,t)
\\
&=&\int_{\mathbb{R}}\int_{\mathbb{R}^{2n-1}}\int_{\mathbb{R}^{2n}}Sc\bigg(\mathcal{SH}_{\psi}^{Q}(f)(a,s,t)\bigg(\int_{\mathbb{R}^{2n}} \psi_{a,s,t}(x)\overline{g(x)}\ dx\bigg)\bigg)\ d\mu(a,s,t)
.\end{eqnarray*}
From Fubini-Tonelli theorem, we get
\begin{eqnarray*}\int_{\mathbb{R}}\int_{\mathbb{R}^{2n-1}}\int_{\mathbb{R}^{2n}}\int_{\mathbb{R}^{2n}}|\mathcal{SH}_{\psi}^{Q}(f)(a,s,t)|g(x)|| \psi_{a,s,t}(x)|\  dx\ d\mu(a,s,t)
\end{eqnarray*}
\begin{eqnarray}\label{inequash}
& = & \int_{\mathbb{R}}\int_{\mathbb{R}^{2n-1}}\int_{\mathbb{R}^{2n}}
|\mathcal{SH}_{\psi}^{Q}(f)(a,s,t)| \ \ \mathcal{SH}_{|\psi|}^{Q}(|g|)(a,s,t)d\mu(a,s,t)\nonumber\\
&\leq & \int_{\mathbb{R}}\int_{\mathbb{R}^{2n-1}}\int_{\mathbb{R}^{2n}}
|\mathcal{SH}_{\psi}^{Q}(f)(a,s,t)|
|\mathcal{SH}_{|\psi|}^{Q}(|g|)(a,s,t)|d\mu(a,s,t)\nonumber \\
 &\leq & \|\mathcal{SH}_{\psi}^{Q}(f)\|_{2,\mu}\|\mathcal{SH}_{|\psi|}^{Q}(|g|)\|_{2,\mu}    \nonumber\\
 & = & \sqrt{C_{\psi}}\sqrt{C_{|\psi|}}\|f\|_{2}\|g\|_{2}<\infty.
 \end{eqnarray}
By Fubini's theorem, we deduce that for every $g\in L^2(\mathbb{R}^{2n},\mathbb{H})$ and for
almost $x\in\mathbb{R}^{2n}$, the function
$$(a,s,t)\mapsto \mathcal{SH}_{\psi}^{Q}(f)(a,s,t)\psi_{a,s,t}(x) \overline{g(x)}$$
is $d\mu$-integrable. Thus, for almost every $x\in\mathbb{R}^{2n}$,
$$(a,s,t)\mapsto \mathcal{SH}_{\psi}^{Q}(f)(a,s,t)\psi_{a,s,t}(x)$$
is $d\mu$-integrable.\\
Let $$\mathcal{H}(x) = \int_{\mathbb{R}}\int_{\mathbb{R}^{2n-1}}\int_{\mathbb{R}^{2n}}\mathcal{SH}_{\psi}^{Q}(f)(a,s,t) \psi_{a,s,t}(x)\ d\mu(a,s,t).$$
The inequality $(\ref{inequash})$ shows that for every $g\in L^2(\mathbb{R}^{2n}, \mathbb{H})$, $g\cdot\mathcal{H}$ belongs to $L^1(\mathbb{R}^{2n},\mathbb{H})$ and consequently $\mathcal{H}$ belongs to $L^2(\mathbb{R}^{2n},\mathbb{H})$. Moreover, from Fubini theorem , we have
\begin{eqnarray*}
C_{\psi}( f,g )_{Q} =\int_{\mathbb{R}^{2n}}Sc(\mathcal{H}(x)\overline{g(x)}) \ dx = ( \mathcal{H},g )_{Q}.
\end{eqnarray*}
This involves that $C_{\psi} f=\mathcal{H}$ in $L^2(\mathbb{R}^{2n},\mathbb{H})$.
\end{proof}
\begin{remark} From Theorem $\ref{thplan}$ (Plancherel's Theorem for $\mathcal{SH}_{\psi}^{Q}$), the function $\mathcal{SH}_{\psi}^{Q}(f)$ belongs to $L^2(d\mu)$ for every $f$ in $L^2(\mathbb{R}^{2n},\mathbb{H})$, hence for
almost every $(a, s)\in \mathbb{R}^*\times\mathbb{R}^{2n-1}$ the function $\mathcal{SH}_{\psi}^{Q}(f)(a,s,.) =(f * \psi_{a,s,0}^{\ast})(.)$ belongs to $L^2(\mathbb{R}^{2n},\mathbb{H})$.
\end{remark}
\begin{theorem} (Lieb inequality for $\mathcal{SH}_{\psi}^{Q}$)\\
Let $\varphi$ and $\psi$ be two admissible quaternion  shearlets satisfies the assumption of theorem \ref{conditions}. For $p \in [1, +\infty[$ and for $f,g \in L^{2}(\mathbb{R}^{2n},\mathbb{H})$, the function
\begin{center}
$(a,s,t) \longmapsto \mathcal{SH}_{\varphi}^{Q}(f)(a,s,t)\mathcal{SH}_{\psi}^{Q}(g)(a,s,t)$
\end{center}
belong to $L^{p}(d\mu)$ and
\begin{center}
$\|\mathcal{SH}_{\varphi}^{Q}(f)\mathcal{SH}_{\psi}^{Q}(g)\|_{p,\mu} \leq \big(\sqrt{C_{\varphi}C_{\psi}}\big)^{\frac{1}{p}} \|f\|_{2}\|g\|_{2} \big(\|\varphi\|_{2}\|\psi\|_{2}\big)^{1-\frac{1}{p}} .$
\end{center}
\end{theorem}
\begin{proof}
\begin{itemize}
\item[(i)]
According to Cauchy-Schwartz inequality and Plancherel's Theorem for the continuous quaternion shearlet transforms $\mathcal{SH}_{\varphi}^{Q}$ and $\mathcal{SH}_{\psi}^{Q}$, for every $f, g \in L^{2}(\mathbb{R}^{2n},\mathbb{H})$,
\begin{eqnarray*}
\displaystyle\iiint_{\mathbb{R}\times\mathbb{R}^{2n-1}\times\mathbb{R}^{2n}} |\mathcal{SH}_{\varphi}^{Q}(f)(a,s,t)\mathcal{SH}_{\psi}^{Q}(g)(a,s,t)|d\mu(a,s,t)
& \leq & \|\mathcal{SH}_{\varphi}^{Q}(f)\|_{2,\mu}\|\mathcal{SH}_{\psi}^{Q}(g)\|_{2,\mu}\\
& = & \sqrt{C_{\varphi}C_{\psi}} \ \|f\|_{2}\|g\|_{2}
\end{eqnarray*}
which implies that $\mathcal{SH}_{\varphi}^{Q}(f)\mathcal{SH}_{\psi}^{Q}(g)$ belongs to $L^{1}(d\mu)$ and
\begin{center}
$\|\mathcal{SH}_{\varphi}^{Q}(f)\mathcal{SH}_{\psi}^{Q}(g)\|_{1,\mu} \leq \sqrt{C_{\varphi}C_{\psi}} \ \|f\|_{2}\|g\|_{2}$.
\end{center}
\item[(ii)]
For every $(a,s,t)\in \mathbb{R}^{*}\times\mathbb{R}^{2n-1}\times\mathbb{R}^{2n}$, using (\ref{psino}), we get
\begin{eqnarray*}
|\mathcal{SH}_{\varphi}^{Q}(f)(a,s,t)\mathcal{SH}_{\psi}^{Q}(g)(a,s,t)| & \leq & \|f\|_{2}\|\varphi_{a,s,t}\|_{2}\|g\|_{2}\|\psi_{a,s,t}\|_{2}\\
& = & \|f\|_{2}\|\varphi\|_{2}\|g\|_{2}\|\psi\|_{2}
\end{eqnarray*}
which implies that $\mathcal{SH}_{\varphi}^{Q}(f)\mathcal{SH}_{\psi}^{Q}(g)$ belongs to $L^{\infty}(d\mu)$ and
\begin{center}
$\|\mathcal{SH}_{\varphi}^{Q}(f)\mathcal{SH}_{\psi}^{Q}(g)\|_{\infty,\mu} \leq  \|f\|_{2}\|g\|_{2} \|\varphi\|_{2}\|\psi\|_{2}$.
\end{center}
\item[(iii)]
For $p \in [1, +\infty[$ we have\\
$\bigg(\displaystyle\iiint_{\mathbb{R}\times\mathbb{R}^{2n-1}\times\mathbb{R}^{2n}} |\mathcal{SH}_{\varphi}^{Q}(f)(a,s,t)\mathcal{SH}_{\psi}^{Q}(g)(a,s,t)|^{p} d\mu(a,s,t)\bigg)^{\frac{1}{p}}$
\begin{eqnarray*}
 & = & \bigg(\displaystyle\iiint_{\mathbb{R}\times\mathbb{R}^{2n-1}\times\mathbb{R}^{2n}} |\mathcal{SH}_{\varphi}^{Q}(f)(a,s,t)\mathcal{SH}_{\psi}^{Q}(g)(a,s,t)|^{p-1+1} d\mu(a,s,t)\bigg)^{\frac{1}{p}}\\
 &\leq &  \|\mathcal{SH}_{\varphi}^{Q}(f)\mathcal{SH}_{\psi}^{Q}(g)\|_{\infty,\mu}^{\frac{p-1}{p}}  \|\mathcal{SH}_{\varphi}^{Q}(f)\mathcal{SH}_{\psi}^{Q}(g)\|_{1,\mu}^{\frac{1}{p}}\\
 &\leq & \big(\|f\|_{2}\|g\|_{2} \|\varphi\|_{2}\|\psi\|_{2}\big)^{\frac{p-1}{p}} \big(\sqrt{C_{\varphi}C_{\psi}} \ \|f\|_{2}\|g\|_{2}\big)^{\frac{1}{p}}\\
 & = & \big(\sqrt{C_{\varphi}C_{\psi}}\big)^{\frac{1}{p}} \|f\|_{2}\|g\|_{2} \big(\|\varphi\|_{2}\|\psi\|_{2}\big)^{1-\frac{1}{p}} .
\end{eqnarray*}
\end{itemize}
\end{proof}
\begin{theorem}\label{thker}
Let $\psi$ be an admissible quaternion  shearlet in $L^2(\mathbb{R}^{2n},\mathbb{H})$ satisfies the assumption of theorem \ref{conditions}. Then, $\mathcal{SH}_{\psi}^{Q}(L^2(\mathbb{R}^{2n},\mathbb{H}))$ is a reproducing kernel Hilbert space in $L^2(d\mu)$, with kernel
\begin{eqnarray}\label{ker}
\mathcal{K}_{\psi}((a,s,t),(a',s',t'))=\frac{1}{C_{\psi}}\langle \psi_{a,s,t}, \psi_{a',s',t'}\rangle _{Q}.
\end{eqnarray}
The kernel $\mathcal{K}_{\psi}$ is point wise bounded, that is\\

$\forall \ (a,s,t),(a',s',t')\in \mathbb{R}^*\times\mathbb{R}^{2n-1}\times\mathbb{R}^{2n}$
\begin{eqnarray}
|\mathcal{K}_{\psi}((a,s,t),(a',s',t'))|\leq\frac{\|\psi\|_{2}^2}{C_{\psi}} .
\end{eqnarray}
\end{theorem}
\begin{proof} Let $\mathcal{K}_{\psi}$ be the Kernel defined on $(\mathbb{R}^*\times\mathbb{R}^{2n-1}\times\mathbb{R}^{2n})^2$ by
\begin{eqnarray}\label{rekernel}
\mathcal{K}_{\psi}((a,s,t),(a',s',t'))=\frac{1}{C_{\psi}}\mathcal{SH}_{\psi}^{Q}(\psi_{a,s,t})(a',s',t').
\end{eqnarray}
From the relation $(\ref{psino})$ and Plancherel relation $(\ref{Ptsh})$, we deduce that for every $(a,s, t)\in \mathbb{R}^*\times\mathbb{R}^{2n-1}\times\mathbb{R}^{2n}$, the function
$$\mathcal{K}_{\psi}((a,s,t),(.,.,.))$$
belongs to $L^2(d\mu)$.\\
Let $F\in \mathcal{SH}_{\psi}^{Q}(L^2(\mathbb{R}^{2n},\mathbb{H}))$,\  $F = \mathcal{SH}_{\psi}^{Q}(f)$,\ $f\in L^2(\mathbb{R}^{2n},\mathbb{H})$. By Definition $\ref{desha}$ and the relation $(\ref{pstsh})$, we get
\begin{eqnarray*}
F(a,s,t)
&=&\mathcal{SH}_{\psi}^{Q}(f)(a,s,t)\\
&=&\langle f, \psi_{a,s,t}\rangle_{Q}\\
&=&\frac{1}{C_{\psi}}\langle \mathcal{SH}_{\psi}^{Q}(f), \mathcal{SH}_{\psi}^{Q}(\psi_{a,s,t})\rangle_{\mu}\\
&=& \langle \mathcal{SH}_{\psi}^{Q}(f), \mathcal{K}_{\psi}((a,s,t),(.,.,.))\rangle_{\mu}
.
\end{eqnarray*}
This shows that $\mathcal{K}_{\psi}$ is a reproducing Kernel for the Hilbert space $\mathcal{SH}_{\psi}^{Q}(L^2(\mathbb{R}^{2n},\mathbb{H}))$.\\
Now, by the relations $(\ref{psino})$, $(\ref{rekernel})$ and Definition $\ref{desha}$, we deduce that for all
$(a,s,t),(a',s',t')\in \mathbb{R}^*\times\mathbb{R}^{2n-1}\times\mathbb{R}^{2n}$

\begin{eqnarray*}
|\mathcal{K}_{\psi}((a,s,t),(a',s',t'))|
&=&\frac{1}{C_{\psi}}\langle \psi_{a,s,t},\psi_{a',s',t'}\rangle_{Q} \\
&\leq & \frac{1}{C_{\psi}}\|\psi_{a,s,t}\|_{2}\|\psi_{a',s',t'}\|_{2}\\
& = & \frac{\|\psi\|_{2}^2}{C_{\psi}} .
\end{eqnarray*}
\end{proof}
\section{Uncertainty principle}
\subsection{Lieb Uncertainty principle for $\mathcal{SH}_{\psi}^{Q}$}
In this subsection, we prove Lieb uncertainty principle for the multivariate continuous  quaternion shearlet transform.
\begin{definition}
\mbox{}\\
A function $f \in L^{2}(\mathbb{R}^{*}\times\mathbb{R}^{2n-1}\times\mathbb{R}^{2n},\mathbb{H})$ is said to be $\xi-$concentrated on a measurable set $\Sigma\subseteq \mathbb{R}^{*}\times\mathbb{R}^{2n-1}\times\mathbb{R}^{2n}$, if
\begin{equation}\label{14}
\bigg(\displaystyle\int_{\Sigma^{c}} |f(a,s,t)|^{2} d\mu(a,s,t)\bigg)^{\frac{1}{2}} \leq \xi\ \|f\|_{2,\mu}.
\end{equation}
If $0 \leq \xi \leq \frac{1}{2}$, then the most of energy is concentrated on $\Sigma$, and $\Sigma$ is indeed the essential support of $f$.\\
If $\xi = 0$, then $\Sigma$ contains the  support of $f$.
\end{definition}
\begin{theorem} (Donoho-Stark for $\mathcal{SH}_{\psi}^{Q}$)\\
Let $\psi$ be a non zero multivariate admissible quaternion shearlet and $f \in L^{2}(\mathbb{R}^{2d},\mathbb{H})$ such that $f \neq 0$. Let $\Sigma$ a measurable set of $\mathbb{R}^{*}\times\mathbb{R}^{2n-1}\times\mathbb{R}^{2n}$ and $\xi \geq 0$.\\
 If  $\mathcal{SH}_{\psi}^{Q}(f)$ is $\xi$-concentrated on $\Sigma$, hence, we have
\begin{equation}
\mu_{l}(\Sigma) \geq  (1-\xi^{2}) \dfrac{C_{\psi}}{\|\psi\|_{2}^{2}}.
\end{equation}
\end{theorem}
\begin{proof}
We have
\begin{eqnarray*}
C_{\psi}\|f\|_{2}^{2} & = & \|\mathcal{SH}_{\psi}^{Q}(f)\|_{2,\mu}^{2}\nonumber\\
& = & \|\chi_{\Sigma^{c}}\mathcal{SH}_{\psi}^{Q}(f)\|_{2,\mu}^{2} + \|\chi_{\Sigma}\mathcal{SH}_{\psi}^{Q}(f)\|_{2,\mu}^{2}\nonumber\\
& \leq & \xi^{2}\|\mathcal{SH}_{\psi}^{Q}(f)\|_{2,\mu}^{2} +  \|\chi_{\Sigma}\mathcal{SH}_{\psi}^{Q}(f)\|_{2,\mu}^{2}\nonumber\\
& = & \xi^{2} C_{\psi}\|f\|_{2}^{2} +  \|\chi_{\Sigma}\mathcal{SH}_{\psi}^{Q}(f)\|_{2,\mu}^{2}.
\end{eqnarray*}
And consequently, we deduce
\begin{eqnarray}\label{17}
C_{\psi}(1-\xi^{2}) \|f\|_{2}^{2}
& \leq & \|\chi_{\Sigma}\mathcal{SH}_{\psi}^{Q}(f)\|_{2,\mu}^{2} \\
& \leq  & \mu_{l}(\Sigma) \|\mathcal{SH}_{\psi}^{Q}(f)\|_{\infty,\mu}^{2} \leq \mu_{l}(\Sigma) \|f\|_{2}^{2}\|\psi\|_{2}^{2}\nonumber
\end{eqnarray}
We may simplify by $\|f\|_{2}^{2}$ to obtain the desired result.
\end{proof}
\begin{theorem} (Lieb uncertainty principle for $\mathcal{SH}_{\psi}^{Q}$)
\mbox{}\\
Let $\psi$ be a non zero multivariate admissible quaternion shearlet and $f \in L^{2}(\mathbb{R}^{2d},\mathbb{H})$ such that $f \neq 0$. Let $\Sigma$ a measurable set of $\mathbb{R}^{*}\times\mathbb{R}^{2n-1}\times\mathbb{R}^{2n}$ and $\xi \geq 0$. If  $\mathcal{SH}_{\psi}^{Q}(f)$ is $\xi$-concentrated on $\Sigma$, hence for every $p > 2$ we have
\begin{equation}\label{16}
\mu(\Sigma) \geq  (1 - \xi^{2})^{\frac{p}{p-2}} \dfrac{C_{\psi}}{\|\psi\|_{2}^{2}}.
\end{equation}
\end{theorem}
\begin{proof}
 $\mathcal{SH}_{\psi}^{Q}(f)$ is $\xi$-concentrated on $\Sigma$, that is to say
\begin{eqnarray*}
\|\chi_{\Sigma} \mathcal{SH}_{\psi}^{Q}(f)\|_{2,\mu}^{2} & = & \displaystyle\iiint_{\mathbb{R}^{*}\times\mathbb{R}^{2d-1}\times\mathbb{R}^{2d}} \chi_{\Sigma}(a,s,t) |\mathcal{SH}_{\psi}^{Q}(f)(a,s,t)|^{2} d\mu(a,s,t)\\
 & \leq & \bigg(\displaystyle\iiint_{\mathbb{R}^{*}\times\mathbb{R}^{2d-1}\times\mathbb{R}^{2d}} \big(\chi_{\Sigma}(a,s,t)\big)^{\frac{p}{p-2}} d\mu(a,s,t)\bigg)^{\frac{p-2}{p}} \\
& \times &  \bigg(\displaystyle\iiint_{\mathbb{R}^{*}\times\mathbb{R}^{2d-1}\times\mathbb{R}^{2d}} |\mathcal{SH}_{\psi}^{Q}(f)(a,s,t)|^{2\frac{p}{2}} d\mu(a,s,t)\bigg)^{\frac{2}{p}}\\
& = & \big(\mu_{l}(\Sigma)\big)^{\frac{p-2}{p}} \times \| \mathcal{SH}_{\psi}^{Q}(f)\|_{p,\mu}^{2}.
\end{eqnarray*}
On the other hand, and again by Lieb inequality, we deduce that
\begin{eqnarray*}
\|\chi_{\Sigma} \mathcal{SH}_{\psi}^{Q}(f)\|_{2,\mu}^{2} & \leq & \big( \mu_{l}(\Sigma)\big)^{\frac{p-2}{p}} \times \| \mathcal{SH}_{\psi}^{Q}(f)\|_{p,\mu}^{2}.\\
& \leq & \big(\mu_{l}(\Sigma)\big)^{\frac{p-2}{p}} C_{\psi}^{\frac{2}{p}} \|f\|_{2}^{2} \|\psi\|_{2}^{2-\frac{4}{p}}\\
\end{eqnarray*}
then, using the relation (\ref{17}), we get
\begin{equation*}
C_{\psi} \|f\|_{2}^{2}  (1 - \xi^{2})\leq \big(\mu_{l}(\Sigma)\big)^{\frac{p-2}{p}} C_{\psi}^{\frac{2}{p}} \|f\|_{2}^{2} \|\psi\|_{2}^{2-\frac{4}{p}}
\end{equation*}
and consequently
$$C_{\psi}^{1-\frac{2}{p}}  \|\psi\|_{2}^{\frac{4}{p}-2}  (1 - \xi^{2})\leq \big(\mu_{l}(\Sigma)\big)^{1-\frac{2}{p}}$$
hence
$$\dfrac{C_{\psi}}{\|\psi\|_{2}^{2}}    (1 - \xi^{2})^{\frac{p}{p-2}} \leq \mu_{l}(\Sigma) .$$
\end{proof}
\subsection{Logarithmic uncertainty principle for $\mathcal{SH}_{\psi}^{Q}$}

The simplest formulation of the uncertainty principle in
harmonic analysis is Heisenberg-Weyl inequality, which gives
us the information that a nontrivial function and its Fourier
transform cannot both be simultaneously sharply localized \cite{log1,log2}. In this section, we first derive a variation on uncertainty
principle associated with the  $\mathcal{SH}_{\psi}^{Q}$.
From this, we establish the logarithmic uncertainty principle which
is valid for the quaternion Fourier transform to the setting of the  $\mathcal{SH}_{\psi}^{Q}$. Due to the uncertainty principle for
QFT , we have the logarithmic uncertainty principle
for the QFT \cite{quaternion} as follows.

\begin{theorem} \label{logarithm}(Logarithmic uncertainty principle for $\mathcal{F}_{Q}$) \\
 For $f$ in $\mathcal{S}(\mathbb{R}^{2n},\mathbb{H})$, we have
\begin{equation}\label{loga}
\displaystyle\int_{\mathbb{R}^{2n}} ln|x| |f(x)|^{2}dx + \displaystyle\int_{\mathbb{R}^{2n}} ln|w| |\mathcal{F}_{Q}(f)(w)|^{2}dw \geq D_{2n} \displaystyle\int_{\mathbb{R}^{2n}}|f(x)|^{2}dx,
\end{equation}
where
\begin{equation}\label{Den}
D_{2n} = \bigg(\dfrac{\Gamma'(\frac{n}{2})}{\Gamma(\frac{n}{2})} + ln(2) \bigg)
\end{equation}
\end{theorem}
\begin{theorem}
Let $\psi$ be an admissible quaternion  shearlet in $L^2(\mathbb{R}^{2n},\mathbb{H})$ satisfying the assumption of theorem \ref{conditions}. Then there exists $D_{2n} > 0$ such that for every $f\in L^2(\mathbb{R}^{2n},\mathbb{H})$, we have
\begin{equation}
\displaystyle\iiint_{\mathbb{R}\times\mathbb{R}^{2n-1}\mathbb{R}^{2n}} ln|t| |\mathcal{SH}_{\psi}^{Q}f(a,s,t)|^{2}d\mu(a,s,t) + C_{\psi} \displaystyle\int_{\mathbb{R}^{2n}} ln|w| |\mathcal{F}_{Q}(f)(w)|^{2}dw  \geq D_{2n}\  C_{\psi}\ \|f\|_{2}^{2}.
\end{equation}
where $D_{2n}$ is given by (\ref{Den}).
\end{theorem}
\begin{proof}
We have\ \\ \ \\
$\displaystyle\iiint_{\mathbb{R}\times\mathbb{R}^{2n-1}\mathbb{R}^{2n}} ln|w| |\mathcal{F}_{Q}(\mathcal{SH}_{\psi}^{Q}(a,s,.))(w)|^{2} \dfrac{da\ ds\ dw}{|a|^{2n+1}}$
\begin{eqnarray*}
&=& |a|^{2-\frac{1}{2n}} \displaystyle\iiint_{\mathbb{R}\times\mathbb{R}^{2n-1}\mathbb{R}^{2n}} ln|w|\  |\mathcal{F}_{Q}(f)(w)|^{2}|\mathcal{F}_{Q}(\psi^{*})(A_{a}^{T}S_{s}^{T}w)|^{2} \dfrac{da\ ds\ dw}{|a|^{2n+1}}\\
&=& \displaystyle\int_{\mathbb{R}^{2n}} ln|w|\  |\mathcal{F}_{Q}(f)(w)|^{2} \displaystyle\iint_{\mathbb{R}\times\mathbb{R}^{2n-1}} |\mathcal{F}_{Q}(\psi^{*})(A_{a}^{T}S_{s}^{T}w)|^{2} \dfrac{da\ ds}{|a|^{\frac{4n^{2}-2n+1}{2n}}}\ dw\\
&=& C_{\psi} \ \displaystyle\int_{\mathbb{R}^{2n}} ln|w|\  |\mathcal{F}_{Q}(f)(w)|^{2} \ dw.
\end{eqnarray*}
We may replace $f$ by $\mathcal{SH}_{\psi}^{Q}$ on both sides of (\ref{loga}) and get
\begin{equation*}
\displaystyle\int_{\mathbb{R}^{2n}} ln|t| |\mathcal{SH}_{\psi}^{Q}f(a,s,t)|^{2}dt + \displaystyle\int_{\mathbb{R}^{2n}} ln|w| |\mathcal{F}_{Q}(\mathcal{SH}_{\psi}^{Q}f(a,s,.))(w)|^{2}dw \geq D_{2n} \displaystyle\int_{\mathbb{R}^{2n}}|\mathcal{SH}_{\psi}^{Q}f(a,s,t)|^{2}dt,
\end{equation*}
integrationg both sided of this equation with respect to $\dfrac{da\ ds}{|a|^{2n+1}}$
 yields\ \\ \ \\
 $\displaystyle\iiint_{\mathbb{R}\times\mathbb{R}^{2n-1}\mathbb{R}^{2n}} ln|t| |\mathcal{SH}_{\psi}^{Q}f(a,s,t)|^{2}d\mu(a,s,t) + C_{\psi} \displaystyle\int_{\mathbb{R}^{2n}} ln|w| |\mathcal{F}_{Q}(f)(w)|^{2}dw $
 \begin{equation*}
\geq D_{2n} \displaystyle\iiint_{\mathbb{R}\times\mathbb{R}^{2n-1}\mathbb{R}^{2n}}|\mathcal{SH}_{\psi}^{Q}f(a,s,t)|^{2}d\mu(a,s,t),
\end{equation*}
we obtain
\begin{equation*}
\displaystyle\iiint_{\mathbb{R}\times\mathbb{R}^{2n-1}\mathbb{R}^{2n}} ln|t| |\mathcal{SH}_{\psi}^{Q}f(a,s,t)|^{2}d\mu(a,s,t) + C_{\psi} \displaystyle\int_{\mathbb{R}^{2n}} ln|w| |\mathcal{F}_{Q}(f)(w)|^{2}dw  \geq D_{2n}\  C_{\psi}\ \|f\|_{2}^{2}.
\end{equation*}
\end{proof}
\begin{corollary}
Let $\psi$ be an admissible quaternion  shearlet in $L^2(\mathbb{R}^{2n},\mathbb{H})$ satisfying the assumption of theorem \ref{conditions}. Then there exists $D_{2d} > 0$ such that for every $f\in L^2(\mathbb{R}^{2n},\mathbb{H})$, we have
\begin{equation}
\displaystyle\iiint_{\mathbb{R}\times\mathbb{R}^{2n-1}\mathbb{R}^{2n}} ln|(a,s,t)| |\mathcal{SH}_{\psi}^{Q}f(a,s,t)|^{2}d\mu(a,s,t) + C_{\psi} \displaystyle\int_{\mathbb{R}^{2n}} ln|w| |\mathcal{F}_{Q}(f)(w)|^{2}dw  \geq D_{2n}\  C_{\psi}\ \|f\|_{2}^{2}.
\end{equation}
where $D_{2n}$ is given by (\ref{Den}) and $|(a,s,t)| = \sqrt{a^{2} + |s|^{2} + |t|^{2}}\geq |t|$ .
\end{corollary}
\subsection{The Beckner's uncertainty principle in terms of entropy for $\mathcal{SH}_{\psi}^{Q}$}
The entropy plays an important role in quantum mechanics and in signal theory, for a better understanding of its physical signification we refer the reader to \cite{defentropy2}. Clearly the entropy represents an advantageous way to measure the decay of a function $f$, so that it was very interesting to localize the entropy of a probability measure and its Fourier transform.\\
The aim of the following is to generalize the localization  of the entropy to the continuous quaternion shearlet transform over the quaternion shearlet plane .
\begin{definition} (Entropie)\cite{defentropy}\\
\mbox{}\\
The entropy of a probability density function $P$ on $\mathbb{R}^{*}\times\mathbb{R}^{2d-1}\times\mathbb{R}^{2d}$ is defined by
$$E(P) = - \displaystyle\iiint_{\mathbb{R}^{*}\times\mathbb{R}^{2d-1}\times\mathbb{R}^{2d}} ln(P(a,s,t)) P(a,s,t)\ d\mu(a,s,t).$$
\end{definition}
The aim of the following is to generalize the localization  of the entropy to the $\mathcal{SH}_{\psi}^{Q}$ over the quaternion shearlet plane .
\begin{theorem}
Let $\psi$ be a non zero multivariate admissible quaternion shearlet in $L^{2}(\mathbb{R}^{2n},\mathbb{H})$ satisfying the assumption of theorem \ref{conditions}. For every function $f$ in $L^2(\mathbb{R}^{2n},\mathbb{H})$ such that $f\neq 0$, we have
\begin{equation}\label{10}
E(|\mathcal{SH}_{\psi}^{Q}(f)|^{2}) \geq C_{\psi} \|f\|_{2}^{2}\ \  ln\bigg(\dfrac{1}{\|f\|_{2}^{2}\ \|\psi\|_{2}^{2}}\bigg) .
\end{equation}
\end{theorem}
\begin{proof}
Assume that $\|f\|_{2} = \|\psi\|_{2} = 1$, we deduce that
$$\forall (a,s,t) \in \mathbb{R}^{*}\times\mathbb{R}^{2d-1}\times\mathbb{R}^{2d}  \hspace{0.1 cm} , \hspace{0.1 cm} |\mathcal{SH}_{\psi}^{Q}(f)(a,s,t)| \leq \|\mathcal{SH}_{\psi}^{Q}(f)\|_{\infty,\mu} \leq \|f\|_{2} \|\psi\|_{2} = 1$$
then $ln (|\mathcal{SH}_{\psi}^{Q}(f)|) \leq 0$ in particular $E(\mathcal{SH}_{\psi}^{Q}(f))\geq 0$.\\
$\bullet$ Therefore if the entropy $E(\mathcal{SH}_{\psi}^{Q}(f)) = +\infty $ then the inequality (\ref{10}) hold trivially .\\
$\bullet$ Suppose now that the entropy $E(\mathcal{SH}_{\psi}^{Q}(f)) < +\infty $  and let $0 < x < 1$ and $H_{x}$ be the function defined on $]2,3]$ by
$$H_{x}(p) = \dfrac{x^{p} - x^{2}}{p - 2}$$
then
\begin{center}
$\forall p \in ]2,3]$, $\dfrac{d}{dp}H_{x}(p) = \dfrac{(p-2) x^{p} ln(x) - x^{p} + x^{2}}{(p-2)^{2}}$
\end{center}
The sign of $\dfrac{dH_{x}}{dp}$ is the same as that of the function $K_{x}(p) = (p-2) x^{p} ln(x) - x^{p} + x^{2}$.\\
For every $0 < x < 1$, the function $K_{x}$ is differentiable on $\mathbb{R}$, especially on $]2,3]$,\\
 whose derivative is
$$\dfrac{d}{dp}K_{x}(p) = (p-2) (ln(x))^{2} x^{p} .$$
We have for all $0< x < 1$, $\dfrac{d}{dp}K_{x}(p)$ is positive on $]2,3]$, then $K_{x}$ is increasing on $]2,3]$.\\
For all $0<x<1$, $\displaystyle\lim_{p \longmapsto 2^{+}} K_{x}(p) = K_{x}(2) = 0$ then $K_{x}$ is positive which implies that $\dfrac{dH_{x}}{dp}$ is positive also on $]2,3]$ and consequently $p \mapsto H_{x}(p)$  is increasing on $]2,3]$.
In particular,
\begin{center}
$\forall p \in ]2,3]$, \hspace{0.5 cm} $x^{2} ln(x) = \displaystyle\lim_{p\longrightarrow 2^{+}} H_{x}(p) \leq \dfrac{x^{p} - x^{2}}{p - 2}$
\end{center}
hence
\begin{equation}\label{11}
\forall p \in ]2,3], \hspace{0.5 cm} 0 \leq \dfrac{ x^{2} - x^{p}}{p - 2} \leq -x^{2} ln(x).
\end{equation}
Inequality (\ref{11}) rest true for $x = 0$ and $x = 1$. Hence for every $0 \leq x \leq 1$ we have
$$\forall p \in ]2,3], \hspace{0.5 cm} 0 \leq \dfrac{x^{2} - x^{p}}{p - 2} \leq -x^{2} ln(x).$$
$ \forall (a,b) \in \mathbb{R}^{*}\times\mathbb{R}\times\mathbb{R}^{2}$, $0\leq |\mathcal{SH}_{\psi}^{Q}(f)(a,s,t)|^{2} \leq 1$, we get for every $p \in ]2,3]$
\begin{equation}\label{12}
0 \leq \dfrac{|\mathcal{SH}_{\psi}^{Q}(f)(a,s,t)|^{2} - \mathcal{SH}_{\psi}^{Q}(f)(a,s,t)|^{p}}{p - 2} \leq - |\mathcal{SH}_{\psi}^{Q}(f)(a,s,t)|^{2} ln(|\mathcal{SH}_{\psi}^{Q}(f)(a,s,t)|).
\end{equation}
Let $\varphi$ be the function defined on $[2, +\infty [$ by
\begin{center}
$\varphi(p) = \bigg(\displaystyle\int_{\mathbb{R}} \displaystyle\int_{\mathbb{R}^{2d-1}} \displaystyle\int_{\mathbb{R}^{2d}} |\mathcal{SH}_{\psi}^{Q}(f)(a,s,t)|^{p} d\mu(a,s,t)\bigg) - C_{\psi} .$
\end{center}
According to Lieb inequality, we know that for every $2 \leq p < +\infty$, $\mathcal{SH}_{\psi}^{Q}(f)$ belongs to $L^{p}(\mathbb{R}^{*}\times\mathbb{R}^{2d-1}\times\mathbb{R}^{2d},\mathbb{H})$ and we have
\begin{equation}\label{13}
\| \mathcal{SH}_{\psi}^{Q}(f)\|_{p,\mu}^{p} \leq C_{\psi} \|f\|_{2}^{p} \|\psi\|_{2}^{p-2}
\end{equation}
Then, relation (\ref{13}) implies that $\varphi(p)\leq 0$ for every $p \in [2, +\infty[$ and by Plancherel theorem we have
\begin{eqnarray*}
\varphi(2) & = & \bigg(\displaystyle\int_{\mathbb{R}} \displaystyle\int_{\mathbb{R}^{2d-1}} \displaystyle\int_{\mathbb{R}^{2d}}|\mathcal{SH}_{\psi}^{Q}(f)(a,s,t)|^{2} d\mu(a,s,t)\bigg) - C_{\psi}\\
& = & \|\mathcal{SH}_{\psi}^{Q}(f)(a,s,t)\|_{2,\mu}^{2} - C_{\psi}\\
& = &  C_{\psi} \|f\|_{2}^{2} - C_{\psi}\\
& = & 0 .
\end{eqnarray*}
Therefore $\bigg(\dfrac{d\varphi}{dp}\bigg)_{p = 2^{+}} \leq 0$ whenever this derivative is well defined. On the other hand, we have for every $p \in ]2,3]$ and for $(a,s,t) \in \mathbb{R}^{*}\times\mathbb{R}^{2d-1}\times\mathbb{R}^{2d}$
\begin{center}
$\bigg|\dfrac{|\mathcal{SH}_{\psi}^{Q}(f)(a,s,t)|^{p} - |\mathcal{SH}_{\psi}^{Q}(f)(a,s,t)|^{2} }{p-2}\bigg| \leq - |\mathcal{SH}_{\psi}^{Q}(f)(a,s,t)|^{2} ln(|\mathcal{SH}_{\psi}^{Q}(f)(a,s,t)|)$.
\end{center}
then\\
$\displaystyle\int_{\mathbb{R}} \displaystyle\int_{\mathbb{R}^{2d-1}} \displaystyle\int_{\mathbb{R}^{2d}}\bigg|\dfrac{|\mathcal{SH}_{\psi}^{Q}(f)(a,s,t)|^{p} - |\mathcal{SH}_{\psi}^{Q}(f)(a,s,t)|^{2} }{p-2}\bigg| d\mu(a,s,t)$
$$\begin{tabular}{lll}
 & $\leq$ & $-\displaystyle\int_{\mathbb{R}} \displaystyle\int_{\mathbb{R}^{2d-1}} \displaystyle\int_{\mathbb{R}^{2d}}|\mathcal{SH}_{\psi}^{Q}(f)(a,s,t)|^{2} ln(|\mathcal{SH}_{\psi}^{Q}(f)(a,s,t)|)d\mu(a,s,t)$\\
& $=$ & $-\dfrac{1}{2} \displaystyle\int_{\mathbb{R}} \displaystyle\int_{\mathbb{R}^{2d-1}} \displaystyle\int_{\mathbb{R}^{2d}} |\mathcal{SH}_{\psi}^{Q}(f)(a,s,t)|^{2} ln(|\mathcal{SH}_{\psi}^{Q}(f)(a,s,t)|^{2})d\mu(a,s,t)$\\
& $=$ & $\ \ \ \dfrac{1}{2} E(|\mathcal{SH}_{\psi}^{Q}(f)(a,s,t)|^{2}) < +\infty.$
\end{tabular}$$
Moreover , for every $p \in ]3, +\infty[$ and for every $(a,s,t) \in \mathbb{R}^{*}\times\mathbb{R}^{2d-1}\times\mathbb{R}^{2d}$, then
$$\bigg|\dfrac{|\mathcal{SH}_{\psi}^{Q}(f)(a,s,t)|^{p} - |\mathcal{SH}_{\psi}^{Q}(f)(a,s,t)|^{2} }{p-2}\bigg| \leq 2|\mathcal{SH}_{\psi}^{Q}(f)(a,s,t)|^{2}$$
and consequently
\begin{equation*}
\displaystyle\int_{\mathbb{R}} \displaystyle\int_{\mathbb{R}^{2d-1}} \displaystyle\int_{\mathbb{R}^{2d}}\bigg|\dfrac{|\mathcal{SH}_{\psi}^{Q}(f)(a,s,t)|^{p} - |\mathcal{SH}_{\psi}^{Q}(f)(a,s,t)|^{2} }{p-2}\bigg| d\mu(a,s,t) \leq 2\|\mathcal{SH}_{\psi}^{Q}(f)\|^{2}_{2,G} = 2. C_{\psi} < +\infty\\
\end{equation*}
However, by using relation (\ref{12}) and Lebesgue's dominated convergence theorem we have\\
$\bigg(\dfrac{d}{dp}\displaystyle\int_{\mathbb{R}} \displaystyle\int_{\mathbb{R}^{2d-1}} \displaystyle\int_{\mathbb{R}^{2d}} |\mathcal{SH}_{\psi}^{Q}(f)(a,s,t)|^{p} d\mu(a,s,t)\bigg)_{p=2^{+}} $
$$\begin{tabular}{lll}
& $=$ & $ \displaystyle\lim_{p\longrightarrow2^{+}} \displaystyle\int_{\mathbb{R}} \displaystyle\int_{\mathbb{R}^{2d-1}} \displaystyle\int_{\mathbb{R}^{2d}} \dfrac{|\mathcal{SH}_{\psi}^{Q}(f)(a,s,t)|^{p}-|\mathcal{SH}_{\psi}^{Q}(f)(a,s,t)|^{2}}{p-2}d\mu(a,s,t) $\\
& $=$ & $ \displaystyle\int_{\mathbb{R}} \displaystyle\int_{\mathbb{R}^{2d-1}} \displaystyle\int_{\mathbb{R}^{2d}} \displaystyle\lim_{p\longrightarrow2^{+}}\dfrac{|\mathcal{SH}_{\psi}^{Q}(f)(a,s,t)|^{p}-|\mathcal{SH}_{\psi}^{Q}(f)(a,s,t)|^{2}}{p-2}d\mu(a,s,t) $\\
& $=$ & $ \displaystyle\int_{\mathbb{R}} \displaystyle\int_{\mathbb{R}^{2d-1}} \displaystyle\int_{\mathbb{R}^{2d}} ln(|\mathcal{SH}_{\psi}^{Q}(f)(a,s,t)|) |\mathcal{SH}_{\psi}^{Q}(f)(a,s,t)|^{2} d\mu(a,s,t)$\\
& $=$ & $ \dfrac{1}{2} \displaystyle\int_{\mathbb{R}} \displaystyle\int_{\mathbb{R}^{2d-1}} \displaystyle\int_{\mathbb{R}^{2d}} ln(|\mathcal{SH}_{\psi}^{Q}(f)(a,s,t)|^{2}) |\mathcal{SH}_{\psi}^{Q}(f)(a,s,t)|^{2} d\mu(a,s,t)$\\
& $=$  & $ - \dfrac{1}{2} E(|\mathcal{SH}_{\psi}^{Q}(f)|^{2})$,
\end{tabular}$$
and consequently
\begin{center}
$\bigg(\dfrac{d\varphi}{dp}\bigg)_{p=2^{+}} = - \dfrac{1}{2} E(|\mathcal{SH}_{\psi}^{Q}(f)|^{2}) - \bigg(\dfrac{d\big(C_{\psi}\big)}{dp}\bigg)_{p=2^{+}}$ .
\end{center}
then
\begin{center}
$\bigg(\dfrac{d\varphi}{dp}\bigg)_{p=2^{+}} = - \dfrac{1}{2} E(|\mathcal{SH}_{\psi}^{Q}(f)|^{2})  \leq 0$
\end{center}
which gives
\begin{center}
$E(|\mathcal{SH}_{\psi}^{Q}(f)|^{2}|) \geq 0.$
\end{center}
So (\ref{10}) is true for $\|f\|_{2} = \|\psi\|_{2} = 1$.\\
For generic, $f, \psi \neq 0$ and let $g = \dfrac{f}{\|f\|_{2}}$ and $\phi = \dfrac{\psi}{\|\psi\|_{2}}$ we get $\|g\|_{2} = \|\phi\|_{2} = 1$ and\\ $E(|\mathcal{SH}_{\phi}^{Q}(g)|^{2}) \geq 0$ . Since
$$\mathcal{SH}_{\phi}^{Q}(g) = \dfrac{\mathcal{SH}_{\psi}^{Q}(f)}{\|f\|_{2}\|\psi\|_{2}}$$
by combining relations Plancherel's formula (\ref{Ptsh}) and Fubini's Theorem we get
$$\begin{tabular}{lll}
\vspace{0.25 cm}
$E(|\mathcal{SH}_{\phi}^{Q}(g)|^{2})$ & $=$ & $- \displaystyle\int_{\mathbb{R}} \displaystyle\int_{\mathbb{R}^{2d-1}} \displaystyle\int_{\mathbb{R}^{2d}}\big(ln(|\mathcal{SH}_{\psi}^{Q}(f)(a,s,t)|^{2})-ln(\|f\|_{2}^{2}\|\psi\|_{2}^{2})\big) \dfrac{|\mathcal{SH}_{\psi}^{Q}(f)(a,s,t)|^{2}}{\|f\|_{2}^{2}\|\psi\|_{2}^{2}} d\mu(a,s,t)$\\
& $=$ & $\dfrac{E(|\mathcal{SH}_{\psi}^{Q}(f)|^{2}) + ln(\|f\|^{2}_{2}\|\psi\|_{2}^{2})\|\mathcal{SH}_{\psi}^{Q}(f)\|^{2}_{2}}{\|f\|^{2}_{2} \|\psi\|^{2}_{2}}$\\
& $=$ & $\dfrac{E(|\mathcal{SH}_{\psi}^{Q}(f)|^{2})}{\|f\|^{2}_{2} \|\psi\|^{2}_{2}} +\dfrac{C_{\psi}}{\|f\|_{2}^{2}} ln(\|f\|^{2}_{2} \|\psi\|^{2}_{2})$\\
& $\geq$ & $0.$
\end{tabular}$$
we deduce that
\begin{center}
$E(|\mathcal{SH}_{\psi}^{Q}(f)|^{2}) \geq C_{\psi}\|f\|_{2}^{2} \ ln\bigg(\dfrac{1}{\|f\|_{2}^{2}\|\psi\|_{2}^{2}}\bigg) .$
\end{center}
\end{proof}



\end{document}